\documentclass{sig-alternate}
\pdfoutput=1

\usepackage{amsmath}    % need for sub equations
\usepackage{amsfonts}
\usepackage{mathabx}
\usepackage{graphicx}   % need for figures
\usepackage{subcaption}
\usepackage{epsfig} 
\usepackage{cancel}
\usepackage{amssymb}
\usepackage{color}
\usepackage{todonotes}

\usepackage{url}
\usepackage{hyperref}
\usepackage[%
    style=numeric-comp,
    sorting=nyt,
    backend=bibtex,
    sortcites=true,
    doi=false,
    firstinits=true,
    hyperref,
    isbn=false,
    eprint=false,
    maxcitenames=3, %How many people before et al.
    block=none]
    {biblatex}
    
    \renewbibmacro{in:}{}
    \AtEveryBibitem{%
  	\clearlist{language}%
	}
      
\DeclareBibliographyCategory{needsurl}
\newcommand{\entryneedsurl}[1]{\addtocategory{needsurl}{#1}}
\renewbibmacro*{url+urldate}{%
  \ifcategory{needsurl}{
    \printfield{url}%
    \iffieldundef{urlyear}
      {}
      {\setunit*{\addspace}%
       \printurldate}}
    {}}

\bibliography{references,library}
\entryneedsurl{Mitchell2004}
\newcommand{\RR}{\mathbb{R}}
\newcommand{\K}{\mathcal{K}}
\newcommand{\X}{\mathcal{X}}
\newcommand{\A}{\mathcal{A}}
\newcommand{\B}{\mathcal{B}}
\newcommand{\T}{\mathcal{T}}
\newcommand{\M}{\mathcal{M}}
\newcommand{\C}{\mathcal{C}}
\renewcommand{\AA}{\mathbb{A}}
\newcommand{\BB}{\mathbb{B}}
\newcommand{\TT}{\mathbb{T}}
\newcommand{\KK}{\mathbb{K}}
\newcommand{\XX}{\mathbb{X}}
\newcommand{\MM}{\mathbb{M}}
\newcommand{\x}{p_x}
\newcommand{\y}{p_y}

\newtheorem{theorem}{Theorem}
\newtheorem{lemma}{Lemma}
\newtheorem{corollary}{Corollary}
\newtheorem{remark}{Remark}
\newtheorem{proposition}{Proposition}

\title{\LARGE \bf
Reach-Avoid Problems with Time-Varying Dynamics, Targets and Constraints}
\author{
\alignauthor
Jaime F. Fisac, Mo Chen, Claire J. Tomlin, and S. Shankar Sastry
\thanks{This work has been supported in part by NSF under CPS:ActionWebs (CNS-931843), by ONR under the HUNT (N0014-08-0696) and SMARTS (N00014-09-1-1051) MURIs and by grant N00014-12-1-0609, by AFOSR under the CHASE MURI (FA9550-10-1-0567). The research of J.F. Fisac has received funding from the ``la Caixa" Foundation.}\\
       \affaddr{Department of Electrical Engineering and Computer Sciences}\\
       \affaddr{University of California, Berkeley}\\
       \affaddr{Berkeley, CA 94720, USA}\\
       \email{\{jfisac, mochen72, tomlin, sastry\}@eecs.berkeley.edu}
}

\begin{document}
\maketitle
\thispagestyle{empty}
\pagestyle{empty}

%%%
\begin{abstract}
We consider a reach-avoid differential game, in which one of the players aims to steer the system into a target set without violating a set of state constraints, while the other player tries to prevent the first from succeeding; the system dynamics, target set, and state constraints may all be time-varying. The analysis of this problem plays an important role in collision avoidance, motion planning and aircraft control, among other applications. Previous methods for computing the guaranteed winning initial conditions and strategies for each player have either required augmenting the state vector to include time, or have been limited to problems with either no state constraints or entirely static targets, constraints and dynamics. To incorporate time-varying dynamics, targets and constraints without the need for state augmentation, we propose a modified Hamilton-Jacobi-Isaacs equation in the form of a double-obstacle variational inequality, and prove that the zero sublevel set of its viscosity solution characterizes the capture basin for the target under the state constraints. Through this formulation, our method can compute the capture basin and winning strategies for time-varying games at no additional computational cost with respect to the time-invariant case. We provide an implementation of this method based on well-known numerical schemes and show its convergence through a simple example; we include a second example in which our method substantially outperforms the state augmentation approach.
\end{abstract}

\section{Introduction}
Dynamic reach-avoid games have received growing interest in recent years and have many important applications in engineering problems, especially concerning the control of strategic or safety-critical systems: in many scenarios, one must find a control action that will guarantee reaching a desired state while respecting a set of constraints, often in the presence of an unknown disturbance or adversary. Practical applications include collision avoidance, surveillance, energy management and safe reinforcement learning, in which targets typically describe desired waypoints or operating conditions and constraints can model obstacles in the environment or generally forbidden configurations.

%Reach-avoid problems have been a subject of active research in the past two decades \cite{Rapaport1998, Quincampoix2002, lygeros04}.

%In the two-person reach-avoid formulation, one of the players (the attacker) tries to drive the system into a target set in the state space within some time horizon, while the other (the defender) tries to prevent it from succeeding. Although in general the outcome of the game can be defined by the closest miss or deepest penetration into the target, or by the minimum time required to reach it, the determining factor is often simply whether or not the system enters the target within the duration of the game.

In the two-player reach-avoid formulation, one seeks to determine the set of states from which one of the players (the attacker) can successfully drive the system to some target set, while keeping the state within some state constraint set at all times, regardless of the opposing actions of the other player (the defender)---this set is commonly referred to as the \emph{capture basin}, \emph{backwards reachable set} or simply \emph{reach-avoid set} of the target under the constraints. In the absence of state constraints, reachability problems involving possibly time-varying target sets can be posed as a maximum (minimum) cost game where the players try to optimize the pointwise minimum over time of some metric to the target. In this case, the backwards reachable set can be obtained by finding the viscosity solution to the corresponding Hamilton-Jacobi-Isaacs (HJI) equation in the form of a variational inequality: this value function captures the minimum distance to the target that will be achieved by the optimal trajectory starting at each point, so the capture basin is characterized by the region of the state space where this minimum future distance is equal to or less than zero. Maximum cost control problems were studied in detail in \cite{Barron1989}, and extended to the two-player setting in \cite{Barron1990}.
While computationally intensive, Hamilton-Jacobi approaches are practically appealing nowadays due to the availability of modern numerical tools such as \cite{Mitchell2005, sethian96, osher03}, which allow solving the associated equations for problems with low dimensionality. %These numerical tools have been successfully used to solve a variety of differential games, path planning problems, and optimal control problems, including aircraft collision avoidance \cite{Mitchell2005}, automated in-flight refueling \cite{ding08}, and general reach-avoid games \cite{huang11}.

If the game is played under state constraints, then the value function generally becomes discontinuous \cite{Rapaport1998, Quincampoix2002}, which leads to numerical issues. In the case of systems with time-invariant dynamics, targets and constraints, the approach in \cite{Bokanowski2010} characterizes the capture basin through an auxiliary value function that solves a modified Hamilton-Jacobi variational inequality. Although the new value function no longer captures the minimum distance from a trajectory to the target, the reach-avoid set is still given by the value function's subzero region. This allows to effectively turn a constrained final cost problem into an unconstrained problem with a maximum cost. %An equivalent formulation is given in \cite{mitchell-thesis}. % I am not stating this until I can be sure that the formulation in Ian's thesis is equivalent.

For problems with time-varying dynamics, targets and constraints, the approach proposed in \cite{Bokanowski2011} as an extension of \cite{Bokanowski2010} requires augmenting the state space with an additional dimension accounting for time; one can then transform time-dependence into state-dependence and apply the above described methods to solve the fixed problem in the space-time state space. Unfortunately, this approach presents a significant drawback, since the complexity of numerical computations is exponential in the problem dimensionality. %Current level set methods are known to scale particularly poorly with the size of the problem and some of the computational tools available \cite{Mitchell2004},\cite{Mitchell2008} are limited in practice to 5 or 6 dimensions. It is therefore desirable to solve the time-varying problem without falling under the curse of dimensionality.

The main contribution of this paper is an extension of the Hamilton-Jacobi reach-avoid formulation to the case where the target set, the state constraint set, and dynamics are allowed to be time-varying, enabling computation of the reach-avoid set at no significant additional cost relative to the time-invariant case. To this end, we formulate a double-obstacle HJI variational inequality, and prove that the zero sublevel set of its viscosity solution characterizes the desired reach-avoid set. %To the best of our knowledge, our work constitutes the first analysis of double-obstacle HJ equations in the context of reachability and reach-avoid problems.
We also provide a numerical scheme based on \cite{osher91, osher03} and implementation based on \cite{Mitchell2004} to solve the variational inequality and verify the numerical solution using a simple example. We finish by showing that our method vastly outperforms techniques requiring state augmentation.

It should be noted that other authors have recently studied Hamilton-Jacobi equations with a double obstacle in the context of games with lack of information \cite{Cardaliaguet2009} and stochastic games with impulse controls \cite{Cosso2013}. To our knowledge, however, the present work constitutes the first analysis of double-obstacle Hamilton-Jacobi equations in the context of reachability problems.
We also note that the results presented in this paper for differential games are readily applicable to optimal control problems.%, by simply setting one player's inputs to a constant value.
%The remainder of the paper is structured as follows: in Section \ref{sec:formulation} we present the general problem formulation for our reach-avoid game as a minimax value problem; in Section \ref{sec:solution} we derive the double-obstacle HJI equation which allows us to solve the time-varying game with no need for state augmentation; in Section \ref{sec:numerical}, we propose a numerical scheme to solve the HJI equation; in Section \ref{sec:example}, we illustrate our proposed technique through a simple numerical example and verify the convergence numerical scheme using an analytic solution to the example.
% Introduction (1-1.5p)
%% Motivation
%% Related work
%% Summary of results

\section{Problem Formulation \label{sec:formulation}}
\subsection{System Dynamics}
Let $\A\subset\RR^{n_a}$ and $\B\subset\RR^{n_b}$ be nonempty compact sets and, for $t\le T$, let $\mathbb{A}_t$ and $\mathbb{B}_t$ denote the collections of measurable\footnote{
A function $f:X\to Y$ between two measurable spaces $(X,\Sigma_X)$ and $(Y,\Sigma_Y)$ is said to be measurable if the preimage of a measurable set in $Y$ is a measurable set in $X$, that is: $\forall V\in\Sigma_Y, f^{-1}(V)\in\Sigma_X$, with $\Sigma_X,\Sigma_Y$ $\sigma$-algebras on $X$,$Y$.
}
functions $a: (t,T)\to \A$ and $b: (t,T)\to \B$ respectively. We consider a dynamical system with state $x\in\RR^n$, and two agents, player I and player II, with inputs $a(\cdot)\in\mathbb{A}_t$ and $b(\cdot)\in\mathbb{B}_t$ respectively. The system dynamics are given by the flow field $f: \RR^n \times \A \times \B\times [0,T]\rightarrow\RR^n$, which is assumed to be measurable in $a,b$ and $t$ for each $x$, and bounded and Lipschitz continuous in $x$ for any $a$, $b$ and $t$; that is,
\begin{align}\label{eq:f}
&\exists L> 0: \forall a\in\A, \forall b\in\B,\forall t\in[0,T],\forall x,\tilde{x}\in\RR^n, \\&|f(x,a,b,t)|<L, |f(x,a,b,t)-f(\tilde{x},a,b,t)|\le L|x-\tilde{x}|,\notag
\end{align}
where $|\cdot|$ denotes a norm on $\RR^n$. Then for any initial time $t\in[0,T]$ and state $x$, under input signals $a(\cdot)\in\mathbb{A}_t$, $b(\cdot)\in\mathbb{B}_t$, the evolution of the system is determined (see for example \cite{Coddington1955}, Chapter 2, Theorems 1.1, 2.1) by the unique continuous trajectory $\phi_{x,t}^{a,b}:[t,T]\to\RR^n$ solving
\begin{equation}\label{eq:xdot}
\begin{split}
\dot{x}(s) &= f(x(s),a(s),b(s),s), \text{ a.e. }s\in[t,T],\\
x(t) &= x.
\end{split}
\end{equation}
Note that this is a solution in the \emph{extended sense}, that is, it satisfies the differential equation \emph{almost everywhere} (i.e. except on a subset of Lebesgue measure zero). It will be useful to denote by $\XX_t$ the collection of all trajectories $\phi_{x,t}^{a,b}:[t,T]\to\RR^n$ that solve \eqref{eq:xdot} for some initial condition and  input signals:\[\begin{split}
\XX&_t :=\big\{\phi_{x,t}^{a,b}:[t,T]\to\RR^n, \text{for }x\in\RR^n,a(\cdot)\in\AA_t,b(\cdot)\in\BB_t \;|\;\\& \phi_{x,t}^{a,b}(t)=x, \frac{d}{dt}\phi_{x,t}^{a,b}=f(\phi_{x,t}^{a,b}(s),a(s),b(s),s)\text{ a.e. }s\in[t,T]\big\}.\end{split}\]

\subsection{Target and Constraint Sets}
Let $t = 0$ be the start time of the reachability game and $t=T>0$ be the end time. Following the notation in \cite{Bokanowski2010,Bokanowski2011}, let $d:\RR^n\times 2^{\RR^n}\to\RR$ give the distance between a point $x$ and a set $\M$ under some norm $|\cdot|$ on $\RR^n$, that is, $d(x,\M) :=\inf_{y\in\M} |x-y|$. Further, for every $\M\in\RR^n$, let $d_\M:\RR^n\to\RR$ be the signed distance function to $\M$:
\[d_\M(x) := \begin{cases}d(x,\M),&x\in\RR^n\setminus\M, \\ -d(x,\RR^n\setminus\M),&x\in\M.\end{cases}\]

A set-valued map $\M:[0,T]\to 2^{\RR^n}$ is upper hemicontinuous\footnote{Sometimes also called upper semicontinuous.} if for any open neighborhood $V$ of $\M(t)$ there is an open neighborhood $U$ of $t$ such that $\M(\tau)\subseteq V$ $\forall\tau\in U$. We define the upper hemicontinuous set-valued maps $\T,\K:[0,T]\to 2^{\RR^n}$ which respectively assign a target set $\T_t\subset\RR^n$ and a constraint set $\K_t\subset\RR^n$ to each time $t\in[0,T]$. Requiring that $\T_t,\K_t$ are closed for all $t$, we can construct the space-time sets
\[\mathbb{T}:= \bigcup_{t\in[0,T]} \T_t\times\{t\}, \qquad\mathbb{K}:=\bigcup_{t\in[0,T]} \K_t\times\{t\},\]
which are then closed subsets of $\RR^n\times[0,T]$ by the following lemma.
\begin{lemma}
Let $\M:[0,T]\to 2^{\RR^n}$ be an upper hemicontinuous set-valued map with $\M(t)=\M_t$ closed in $\RR^n$ for all $t\in[0,T]$. Then the set $\MM=\bigcup_{t\in[0,T]} \M_t\times\{t\}$ is closed in $\RR^n\times[0,T]$.
\begin{proof}
We prove the lemma by contradiction using the following characterization: a set is closed if and only if it contains all of its limit points. Suppose $\MM$ is not closed: then there exists a limit point $(x,t)$ of $\MM$ that is not in $\MM$, i.e. $x\not\in\M_t$. Since $\M_t$ is closed, $x\not\in\M_t$ cannot be a limit point of $\M_t$, that is, under any metric on $\RR^n$, there exists $r>0$ such that $\M_t$ and the open ball $B(x,r)$ are disjoint. On the other hand, as $(x,t)$ is a limit point of $\MM$, every open neighborhood of $(x,t)$ must meet $\MM$. In particular, any neighborhood of the form $B(x,r/2)\times U$, with $U$ any open neighborhood of $t$, contains a point $(y,s)\in\MM$, i.e. $y\in\M_s$. But then for the open neighborhood of $\M_t$ given by the Minkowski sum $V:=\M_t + B(0,r/2)$, we have that for all open neighborhoods $U$ of $t$, $\exists s\in U$ such that there is $y\in\M_s\cap B(x,r/2)$, so $y\not\in V$ and therefore $\M_s\not\subseteq V$. This directly contradicts upper hemicontinuity of the set-valued map $\M$.
\end{proof}
\end{lemma}
The closed sets $\TT$ and $\KK$ can then be implicitly characterized as the subzero regions of two Lipschitz continuous functions $l:\RR^n\times[0,T]\rightarrow\RR$ and $g:\RR^n\times[0,T]\rightarrow\RR$ respectively, that is, $\exists L_l,L_g>0: \forall (x,t),(\tilde{x},\tilde{t})\in\RR^n\times [0,T]$,
\begin{align}\label{eq:lg}
|l(x,t) - l(\tilde{x},\tilde{t})|&\le L_l | (x,t) - (\tilde{x},\tilde{t})|, \notag
\\|g(x,t) - g(\tilde{x},\tilde{t})|&\le L_g | (x,t) - (\tilde{x},\tilde{t})|,
\end{align}
so that
\[(x,t)\in\TT\iff l(x,t)\le0,\qquad(x,t)\in\KK\iff g(x,t)\le0.\]
These functions always exist, since we can simply choose the signed distance functions $l(x,t) = d_{\mathbb{T}}(x,t)$ and $g(x,t) = d_{\mathbb{K}}(x,t)$, which are Lipschitz continuous by construction, i.e. they are the infimum of point-to-point distances.

Note that this definition of targets and constraints is flexible and allows to formulate a variety of target and constraint behaviors, including changing topologies over time (e.g. a target splitting into multiple separate sets or disappearing entirely).

We say that a trajectory $\phi\in\mathbb{X}_t$ is \emph{admissible} on $[t,t+\delta]$ for some $\delta>0$ if for all $t\le\tau\le t+\delta$, it satisfies $\phi(\tau)\in\K_\tau$. The minimum value of $l$ achieved by an admissible state trajectory in the course of the game determines its outcome (it will be negative if the trajectory ever enters the target $\T_t$); we therefore refer to $l$ as the \emph{payoff function}. On the other hand, the maximum value of $g$ reached by any trajectory determines whether or not it is admissible (it will be positive if the trajectory ever breaches the constraints $\K_t$); we call $g$ the \emph{discriminator function}.

%%% Note we got rid of the constraint $l(x,t)\ge g(x,t)$. %%%
%Since we only care about reaching the target through admissible trajectories (i.e. trajectories that remain within the constraint set), we will in fact restrict $l(x,t)$ to be greater or equal to $g(x,t)$ at every point: this has the effect that a state inside the target but outside of the constraints will not be assigned a negative payoff. Note that choosing $l(x,t) := \max\{d_{\T_t}(x),d_{\K_t}(x)\}$ also leads to a Lipschitz continuous payoff function.

\subsection{Value and Strategies}\label{Info}

We will adopt the arbitrary convention that player I seeks to minimize the outcome of the game, while player II tries to maximize it: that is, I is trying to drive the system into the target set, and II wants to prevent I from succeeding, possibly by driving the system out of the constraint set; we will refer to I as the \emph{attacker} and II as the \emph{defender}. For each trajectory $\phi_{x,t}^{a,b}\in\XX_t$ we define the outcome of the game as the functional
\begin{small}
\begin{equation}\label{eq:V}
\mathcal{V}\big(x,t,a(\cdot),b(\cdot)\big) = \min_{\tau\in[t,T]}\max\left\{l(\phi_{x,t}^{a,b}(\tau),\tau),\max_{s\in[t,\tau]} g(\phi_{x,t}^{a,b}(s),s)\right\}.
\end{equation}
\end{small}\noindent
%where $\phi_x^*(\cdot)$ refers to the trajectory followed by the system starting at state $x$ and time $t$ when both players play their optimal inputs $a^*(\cdot), b^*(\cdot)$ throughout the interval $[t,T]$ under a given information pattern.
The above expression is considering, for each time $\tau$, the maximum between the current value of $l$ and the greatest value of $g$ reached so far by the trajectory; therefore this term will be smaller of equal to zero for a given $\tau$ if and only if the system is in the target at time $\tau$ without ever having left the constraint set on $[t,\tau]$. If this situation takes place for any $\tau\in[t,T]$, player I wins the game; therefore, the minimum for all $\tau$ reflects whether player I wins at any point between $t$ and the end of the game. We summarize this through the following proposition.
\begin{proposition}\label{Value}
The set of points $x$ at time $t\in[0,T]$ from which the system trajectory $\phi^{a,b}_{x,t}(\cdot)$ under given controls $a(\cdot)\in\mathbb{A}_t,b(\cdot)\in\mathbb{B}_t$ will enter the target set at some time $\tau\in[t,T]$ without violating the constraints at any $s\in[t,\tau]$ is equal to the zero sublevel set of $\mathcal{V}\big(x,t,a(\cdot),b(\cdot)\big)$. That is:
\begin{align}
\{&(x,t)\in\RR^n\times[0,T]: \exists \tau\in[t,T],\notag\\
&\qquad \phi^{a,b}_{x,t}(\tau)\in\T_\tau\; \wedge\; \forall s\in[t,\tau],\phi^{a,b}_{x,t}(s)\in\K_s\}\notag\\
&= \{(x,t)\in\RR^n\times[0,T]: \mathcal{V}\big(x,t,a(\cdot),b(\cdot)\big)\le0\}.
\end{align}
\end{proposition}
Note that this value function is negative when the trajectory starting at $(x,t)$ reaches the target without \emph{previously} breaching the constraints: it is agnostic to whether constraints are breached \emph{after} the target has been reached. One could formulate an alternative problem requiring that trajectories remain feasible for the entire duration of the game: in that case the game's outcome would instead be
\begin{small}
\begin{equation}\label{eq:W}
\mathcal{W}\big(x,t,a(\cdot),b(\cdot)\big) =\max\left\{ \min_{\tau\in[t,T]}l(\phi_{x,t}^{a,b}(\tau),\tau),\max_{\tau\in[t,T]} g(\phi_{x,t}^{a,b}(s),s)\right\}.
\end{equation}
\end{small}\noindent
%smaller formula to fit in line
This alternative problem is not the object of this paper, and we will restrict our attention to the problem described by \eqref{eq:V}. %\todo[size=8]{Expand?}

Following \cite{Varaiya1967,Roxin1969,Elliott1972,Evans1984}, we define the set of nonanticipative strategies for player II containing the functionals\\$\Lambda_t = \{\beta:\mathbb{A}_t\to\mathbb{B}_t\;|\;\forall s\in[t,T],\; \forall a(\cdot),\hat{a}(\cdot)\in\mathbb{A}_t,\; \big(a(\tau) = \hat{a}(\tau)\text{ a.e.} \tau\in[t,s]\big)\Rightarrow \big(\beta[a](\tau) = \beta[\hat{a}](\tau)\text{ a.e.} \tau\in[t,s]\big)\}$. By allowing II to use nonanticipative strategies, we are giving it a certain advantage, since at each instant it can adapt its control input to the one declared by I. This information pattern leads to the upper value of the game, given by:
\begin{subequations}\label{eq:Values}
\begin{align}
&\label{eq:Upper}
{V}^+(x,t):=\sup_{\beta(\cdot)\in\Lambda_t}\inf_{a(\cdot)\in\mathbb{A}_t}\mathcal{V}\big(x,t,a(\cdot),\beta[a](\cdot)\big).
\intertext{Analogously, we can decide to give I the advantage by defining its set of nonanticipative strategies as $\Gamma_t := \{\alpha: \mathbb{B}_t\to\mathbb{A}_t \;|\; \forall s\in[t,T],\;\forall b(\cdot),\hat{b}(\cdot)\in\mathbb{B}_t,\; \big(b(\tau) = \hat{b}(\tau)\text{ a.e.} \tau\in[t,s]\big) \Rightarrow \big(\alpha[b](\tau) = \alpha[\hat{b}](\tau)\text{ a.e.} \tau\in[t,s]\big)\}$. This determines the lower value of the game as:}
&\label{eq:Lower}
{V}^-(x,t):=\inf_{\alpha(\cdot)\in\Gamma_t}\sup_{b(\cdot)\in\mathbb{B}_t}\mathcal{V}\big(x,t,\alpha[b](\cdot),b(\cdot)\big).\end{align}
\end{subequations}
Naturally, it follows that ${V}^-(x,t)\le{V}^+(x,t)$ everywhere. In those cases in which equality holds, the game is said to have value and $V(x,t) := {V}^+(x,t)={V}^-(x,t)$ is simply referred to as the \emph{value of the game}.

%\begin{definition}
Given an information pattern, we say that a point $(x,t)$ is in the \emph{capture basin} (or \emph{reach-avoid set}) $\C_{\TT,\KK}$ of the target $\TT$ under constraints $\KK$ when the system trajectory $\phi^{a,b}_{x,t}$, with both players acting optimally, reaches $\TT$ at some time $\tau\in[t,T]$ while remaining in $\KK$ for all time $s\in[t,\tau]$. In particular, when player II uses nonanticipative strategies,
\begin{subequations}
\begin{align}
\C&_{\TT,\KK}^+ := 
\{(x,t)\in\RR^n\times[0,T]: \exists a(\cdot)\in\AA_t, \forall \beta(\cdot)\in\Gamma_t, \\
&\exists \tau\in[t,T], \phi^{a,\beta[a]}_{x,t}(\tau)\in\T_\tau\; \wedge\; \forall s\in[t,\tau],\phi^{a,\beta[a]}_{x,t}(s)\in\K_s\},\notag
\end{align}
and similarly, when player I uses nonanticipative strategies,
\begin{align}
\C&_{\TT,\KK}^- := 
\{(x,t)\in\RR^n\times[0,T]: \exists \alpha(\cdot)\in\Lambda_t, \forall b(\cdot)\in\BB_t, \\
&\exists \tau\in[t,T], \phi^{\alpha[b],b}_{x,t}(\tau)\in\T_\tau\; \wedge\; \forall s\in[t,\tau],\phi^{\alpha[b],b}_{x,t}(s)\in\K_s\}.\notag
\end{align}
\end{subequations}
%\end{definition}
Given Proposition \ref{Value} and the above definitions, we have an important result expressed by the following proposition.
\begin{proposition}
The capture basin of the space-time target set $\mathbb{T}$ when the defender (attacker) is allowed to use nonanticipative strategies is given by the zero sublevel set of the upper (resp. lower) value function $V^\pm$. That is:
\begin{subequations}
\begin{align}
\C_{\TT,\KK}^+&= \{(x,t)\in\RR^n\times[0,T]: V^+(x,t)\le0\},\\
\C_{\TT,\KK}^-&= \{(x,t)\in\RR^n\times[0,T]: V^-(x,t)\le0\}.
\end{align}
\end{subequations}
\end{proposition}
\begin{corollary}
The capture basin when the defender is allowed to use nonanticipative strategies is a subset of that resulting from the attacker using nonanticipative strategies:
\begin{equation}\C_{\TT,\KK}^+ \subseteq \C_{\TT,\KK}^-.\end{equation}
\end{corollary}
% The following remark is not trivial, Lipschitz continuity of V actually follows from the variational inequality!
%\begin{remark}
%The value functions $V^\pm$ are Lipschitz continuous by construction, as a result of the Lipschitz continuity of $l$ and $g$ and boundedness of $f$: in particular, from \eqref{eq:f} and \eqref{eq:lg} we have:
%\[|V^\pm(x,t) - V^\pm(\tilde{x},\tilde{t})|&\le L_l | (x,t) - (\tilde{x},\tilde{t})|\]
%\end{remark}
% Problem formulation (1p)
%% Cost function
%% Reachable set interpretation

\section{The Double-Obstacle Isaacs \\Equation}\label{Results}
%
%
% ------------------------------------- Dynamic Programming Principle Equations on top
\begin{figure*}[ht]
\begin{subequations}\label{eq:DPP}
\begin{align}
V(x,t)^+ = \sup_{\beta\in\Lambda_t}\inf_{a\in\mathbb{A}_t}\Bigg\{ &\min\bigg[\min_{\tau\in[t,t+\delta]} \max\Big(l(\phi_{x,t}^{a,\beta[a]}(\tau),\tau),\max_{s\in[t,\tau]}g(\phi_{x,t}^{a,\beta[a]}(s),s)\Big),\notag\\
&\max\Big(V(\phi_x^{a,\beta[a]}(t+\delta),t+\delta),\max_{\tau\in[t,t+\delta]}g(\phi_{x,t}^{a,\beta[a]}(\tau),\tau)\Big)\bigg] \Bigg\}.\label{eq:DPP_upper}\\
\notag\\
V(x,t)^- = \inf_{\alpha\in\Gamma_t}\sup_{b\in\mathbb{B}_t}\Bigg\{ &\min\bigg[\min_{\tau\in[t,t+\delta]} \max\Big(l(\phi_{x,t}^{a,\beta[a]}(\tau),\tau),\max_{s\in[t,\tau]}g(\phi_{x,t}^{a,\beta[a]}(s),s)\Big),\notag\\
&\max\Big(V(\phi_x^{a,\beta[a]}(t+\delta),t+\delta),\max_{\tau\in[t,t+\delta]}g(\phi_{x,t}^{a,\beta[a]}(\tau),\tau)\Big)\bigg] \Bigg\}.\label{eq:DPP_lower}
\end{align}
\end{subequations}
\end{figure*}
% -----------------------------------------------------------------------------------------------------------
%
%
It has been shown that the value function for minimum payoff games can be characterized as the unique viscosity solution to a variational inequality involving an appropriate Hamiltonian \cite{Barron1989,Barron1990}, which has commonly been referred to as a Hamilton-Jacobi equation with an obstacle. We now extend the results for minimum cost problems to the category of problems with a cost in the form of \eqref{eq:V}. %Since this cost reflects the minimum over time of a function involving a running maximum over time, it is appropriate to refer to this as a \emph{minimax cost problem} (although this minimax should not be confused with that resulting from the competition between the two players).

We first state the particular form of Bellman's principle of optimality \cite{Bellman1957} for the problem at hand.%minimax cost problem.
\begin{lemma}[Dynamic Programming Principle]\label{DPP}Let\\ $0\le t < T$ and $0<\delta\le T-t$. Then equation \eqref{eq:DPP} holds.

%\begin{small}\begin{subequations}\label{eq:DPP}
%\begin{align}
%V(x,t)^+ = \sup_{\beta\in\Lambda_t}\inf_{a\in\mathbb{A}_t}\Bigg\{ &\min\bigg[\min_{\tau\in[t,t+\delta]} \max\Big(l(\phi(\tau),\tau),\max_{s\in[t,\tau]}g(\phi(s),s)\Big),\notag\\
%&\max\Big(V(\phi_x^{a,\beta[a]}(t+\delta),t+\delta),\max_{\tau\in[t,t+\delta]}g(\phi(\tau),\tau)\Big)\bigg] \Bigg\}.\label{eq:DPP_upper}\\
%\notag\\
%V(x,t)^- = \inf_{\alpha\in\Gamma_t}\sup_{b\in\mathbb{B}_t}\Bigg\{ &\min\bigg[\min_{\tau\in[t,t+\delta]} \max\Big(l(\phi(\tau),\tau),\max_{s\in[t,\tau]}g(\phi(s),s)\Big),\notag\\
%&\max\Big(V(\phi_x^{a,\beta[a]}(t+\delta),t+\delta),\max_{\tau\in[t,t+\delta]}g(\phi(\tau),\tau)\Big)\bigg] \Bigg\}.\label{eq:DPP_lower}
%\end{align}
%\end{subequations}\end{small}\noindent

%\begin{small}
%\end{small}
%\begin{proof}
%The proof follows classical results and is omitted; a similar proof is provided in Theorem 1.2 of \cite{Barron1990}.%\end{proof}
\end{lemma}
%\textcolor{red}
%{The proof follows classical results and is omitted; a similar proof is provided in Theorem 1.2 of \cite{Barron1990}.}
\begin{proof} The correctness of this lemma can be verified by inspection of \eqref{eq:DPP}, considering how the value in \eqref{eq:V} is propagated back in time as per \eqref{eq:Values} along the \emph{characteristic} (optimal trajectory) in all possible cases. The first term in the outer minimum of \eqref{eq:DPP} is the local application of the definition in \eqref{eq:V} restricted to the interval $[t,t+\delta]$,
\begin{equation}
V_{[t,t+\delta]}:=\!\!\min_{\tau\in[t,t+\delta]} \!\!\max\Big(l(\phi_{x,t}^{a,b}(\tau),\tau),
\max_{s\in[t,\tau]}g(\phi_{x,t}^{a,b}(s),s)\Big).
\end{equation}
The minimum outcome achieved in the whole of $[t,T]$, however, will also be a function of the future value of \eqref{eq:V} throughout the remainder of the game after $[t+\delta]$, captured by
\begin{equation}
V_{[t+\delta,T]}:=V(\phi_{x,t}^{a,b}(t+\delta),t+\delta).
\end{equation}
Now, if for all $\tau\in[t,t+\delta]$, $g(\phi_{x,t}^{a,b}(\tau),\tau)\le V_{[t+\delta,T]}$, then from \eqref{eq:V} it will clearly be that $V(x,t) = \min\{V_{[t,t+\delta]},V_{[t+\delta,T]}\}$. The future value that is propagated along the characteristic, however, will be altered if anywhere on $[t,t+\delta]$, $g$ exceeds $V_{[t+\delta,T]}$, in which case the maximum of $g$ along the characteristic between $t$ and $t+\delta$ will be propagated instead. Thus the second term in the outer minimum of \eqref{eq:DPP} is %not $V_{[t+\delta,T]}$, but
\begin{equation}
V_{t\leftarrow[t+\delta,T]}:=\max\big(V_{[t+\delta,T]},\max_{\tau\in[t,t+\delta]}g(\phi_{x,t}^{a,b}(\tau),\tau)\big).
\end{equation}
The resulting value at $(x,t)$ is therefore determined by the minimum of the local element and this last term, that is:
\begin{subequations}\label{eq:DPP_compact}
\begin{align}
&V^+(x,t) =  \sup_{\beta\in\Lambda_t}\inf_{a\in\mathbb{A}_t}\min\{V_{[t,t+\delta]},V_{t\leftarrow[t+\delta,T]}\},\\
& V^-(x,t) = \inf_{\alpha\in\Gamma_t}\sup_{b\in\mathbb{B}_t}\min\{V_{[t,t+\delta]},V_{t\leftarrow[t+\delta,T]}\}.
\end{align}
\end{subequations}
The statement in \eqref{eq:DPP_compact} is a more compact form of \eqref{eq:DPP}.\qed
\end{proof}
%\todo[inline]{Prove DPP? - 2 pages in Barron!}
We introduce the \emph{upper} and \emph{lower Hamiltonians} $H^\pm$:
\begin{subequations}\label{eq:Hamiltonians}
\begin{align}
&H^+(x,p,t) = \min_{a\in\A}\max_{b\in\B} \;f(x,a,b,t)\cdot p,\\
&H^-(x,p,t) = \max_{b\in\B}\min_{a\in\A} \;f(x,a,b,t)\cdot p.
\end{align}
\end{subequations}

The following theorem constitutes the main theoretical contribution of this paper; it shows that the value function $V^\pm$ is the viscosity solution of a particular variational inequality that has the form of a Hamilton-Jacobi-Isaacs equation with a double obstacle.

\begin{theorem}\label{DOHJI}
Assume $f$ satisfies \eqref{eq:f}, and that $l(x,t),g(x,t)$ are globally Lipschitz continuous. Then the value function ${V}^\pm(x,t)$ for the game with outcome given by \eqref{eq:V} is the unique viscosity solution of the variational inequality
%\begin{figure*}[ht]
\begin{subequations}\label{eq:HJI}
\begin{align}
\max\bigg\{&\min\Big\{\partial_t V + H^\pm\left(x, D_x V,t\right),l(x,t)-V(x, t)\Big\},\notag\\
& g(x,t)-V(x,t)\bigg\}=0, \quad t\in[0,T], x\in\RR^n,\label{eq:HJI_variational}\\
\intertext{with terminal condition}
&V(x,T) = \max\big\{l(x,T),g(x,T)\big\},  \quad x\in\RR^n.\label{eq:HJI_boundary}
\end{align}
\end{subequations}
%\end{figure*}
\end{theorem}
\begin{proof}[of Theorem \ref{DOHJI}]
The structure of the proof follows the classical approach in \cite{Evans1984} and draws from viscosity solution theory. In every case, we start by assuming that $V^\pm$ is not a viscosity solution of the HJI equation and derive a contradiction of the principle of optimality stated in Lemma \ref{DPP}. We will prove the theorem for $V^+$ with Hamiltonian $H^+$; the proof for $V^-$ with $H^-$ is analogous and is not presented here.

First, from the definition of $V^+$ in \eqref{eq:V},\eqref{eq:Values}, considering the terminal case $t=T$, it is clear that it satisfies the boundary condition \eqref{eq:HJI_boundary}.

A continuous function is a viscosity solution of a partial differential equation if it is both a \emph{subsolution} and a \emph{supersolution} (defined below). We will first prove that $V^+$ is a viscosity subsolution of \eqref{eq:HJI_variational}. Let $\psi\in C^1(\RR^n\times(0,T))$ such that $V^+ - \psi$ attains a local maximum at $(x_0,t_0)$; without loss of generality, assume that this maximum is 0. We say that $V^+$ is a subsolution of \eqref{eq:HJI_variational} if, for any such $\psi$,
\begin{small}
%\begin{subequations}
\begin{align}\label{eq:subsolution}
\max\bigg\{&\min\Big\{\partial_t \psi(x_0, t_0) + H^+\left(x_0, D_x \psi,t_0\right),l(x_0,t_0)-\psi(x_0, t_0)\Big\},\nonumber\\&g(x_0,t_0)-\psi(x_0, t_0)\bigg\} \ge 0.
\end{align}
%\end{subequations}
\end{small}\noindent
From the condition of local maximum, we know
\[\label{eq:local_max}V^+(\phi_{x_0,t_0}^{a,b}(t_0+\delta),t_0+\delta) \le \psi(\phi_{x_0,t_0}^{a,b}(t_0+\delta),t_0+\delta),\]
for sufficiently small $\delta>0$ and all $a(\cdot)\in\mathbb{A}_{t_0},b(\cdot)\in\mathbb{B}_{t_0}$. For conciseness, we will write $\phi_{x_0,t_0}^{a,b}(t_0+\delta)$ as simply $\phi(t_0+\delta)$ whenever statements hold for all inputs $a(\cdot),b(\cdot)$.

For the sake of contradiction, suppose \eqref{eq:subsolution} is false. Then it must be that
\begin{equation}\label{eq:not_sub_1}
g(x_0,t_0) = \psi(x_0,t_0) - \theta_1,\qquad\qquad\qquad % Need to align with subequations below
\end{equation}
and, in addition, at least one of the following holds:
\begin{subequations}\label{eq:not_sub_2}
\begin{align}
& l(x_0,t_0) = \psi(x_0,t_0)-\theta_2\label{eq:not_sub_2a},\\
&\partial_t\psi(x_0,t_0) + H^+\left(x_0,D_x\psi,t_0\right) = -\theta_3 \label{eq:not_sub_2b},
\end{align}
\end{subequations}
for some $\theta_1,\theta_2,\theta_3>0$. If \eqref{eq:not_sub_1} and \eqref{eq:not_sub_2a} are true, then by continuity of $g$, $l$ and system trajectories, there exists a sufficiently small $\delta>0$ such that for all inputs $a(\cdot)$,$b(\cdot)$ and for all $t_0\le\tau\le t_0+\delta$,
\[\begin{split}
g(\phi(\tau), \tau) &\le \psi(x_0,t_0) - \frac{\theta_1}{2} = V^+(x_0,t_0) - \frac{\theta_1}{2},\\
l(\phi(\tau), \tau) &\le\psi(x_0,t_0) - \frac{\theta_2}{2} = V^+(x_0,t_0) - \frac{\theta_2}{2}.
\end{split}
\]
Then, incorporating this into the dynamic programming principle \eqref{eq:DPP_upper} we have
\begin{small}
\[\begin{aligned}V^+(x_0,t_0) &\le \sup_{\beta\in\Lambda_t}\inf_{a\in\mathbb{A}_t}\Big\{&&\hspace{-2.5cm} \min_{\tau\in[t_0,t+\delta]} \max\Big[l(\phi_{x_0,t_0}^{a,\beta[a]}(\tau),\tau),\\&&&\hspace{-1.2cm}\max_{s\in[t_0,\tau]}g(\phi_{x_0,t_0}^{a,\beta[a]}(s),s)\Big]\Big\}\\
&\le V^+(x_0,t_0) - \min\Big\{\frac{\theta_1}{2},\frac{\theta_2}{2}\Big\},
\end{aligned} \]
\end{small}\noindent
which is a contradiction, since $\theta_1,\theta_2>0$.

Similarly, if \eqref{eq:not_sub_1} and \eqref{eq:not_sub_2b} are true, then for a small enough $\delta>0$ we have, for all $t_0\le\tau\le t_0 + \delta$ and for all nonanticipative strategies $\beta(\cdot)$,
\[\partial_t\psi(\phi_{x_0,t_0}^{a,\beta[a]}(\tau),\tau) + H^+\left(\phi_{x_0,t_0}^{a,\beta[a]}(\tau),D_x\psi,\tau\right) \le -\frac{\theta_3}{2},\]
for some input $a\in\mathbb{A}(t_0)$. Integrating on $[t_0,t_0+\delta]$ gives
\[\psi(\phi(t_0+\delta), t_0+\delta) - \psi(x_0,t_0) \le - \frac{\theta_3}{2}\delta,\]
and recalling that $V^+ - \psi$ has a local maximum at $(x_0,t_0)$, we obtain
\[V^+(\phi(t_0+\delta), t_0+\delta) \le V^+(x_0,t_0) - \frac{\theta_3}{2}\delta.\]

Inspecting \eqref{eq:DPP_upper} in this case, we obtain
\begin{small}
\[\begin{aligned}V^+(x_0,t_0) &\le \sup_{\beta\in\Lambda_t}\inf_{a\in\mathbb{A}_t}\Big\{&&\hspace{-2.6cm}\max \Big[  V^+(\phi_{x_0,t_0}^{a,\beta[a]}(t_0+\delta),t_0+\delta),\\&&&\hspace{-1.8cm}\max_{\tau\in[t_0,t_0+\delta]}g(\phi_{x_0,t_0}^{a,\beta[a]}(\tau),\tau)\Big]\Big\}\\
&\le V^+(x_0,t_0) - \min\Big\{\frac{\theta_1}{2},\frac{\theta_3}{2}\delta\Big\},
\end{aligned} \]
\end{small}\noindent
which again is a contradiction, since $\theta_1,\theta_3,\delta>0$. Therefore, we conclude that \eqref{eq:subsolution} must be true and hence $V^+$ is indeed a subsolution of \eqref{eq:HJI_variational}.

We now proceed to show that $V^+$ is also a viscosity supersolution of \eqref{eq:HJI_variational}, that is, for all $\psi\in C^1(\RR^n\times(0,T))$ such that $V^+ - \psi$ attains a local minimum at $(x_0,t_0)$ (again, we can assume for convenience that this minimum is 0), it holds that
\begin{small}\begin{align}\label{eq:supersolution}
\max\bigg\{&\min\Big\{\partial_t \psi(x_0, t_0) + H^+\left(x_0, D_x \psi,t_0\right),l(x_0,t_0)-\psi(x_0, t_0)\Big\},\notag\\&g(x_0,t_0)-\psi(x_0, t_0)\bigg\} \le 0.
\end{align}\end{small}\noindent
If we suppose that \eqref{eq:supersolution} is false, then either it holds that
\begin{equation}\label{eq:not_super_1}
g(x_0,t_0) = \psi(x_0,t_0) + \theta_1,\qquad\qquad\qquad % Need to align with subequations below
\end{equation}
or both of the following are true:
\begin{subequations}\label{eq:not_super_2}
\begin{align}
& l(x_0,t_0) = \psi(x_0,t_0)+\theta_2\label{eq:not_super_2a},\\
&\partial_t\psi(x_0,t_0) + H^+\left(x_0,D_x\psi,t_0\right) = \theta_3 \label{eq:not_super_2b},
\end{align}
\end{subequations}
for some $\theta_1,\theta_2,\theta_3>0$.

If \eqref{eq:not_super_1} holds, then there is a small enough $\delta>0$ such that for all trajectories starting at $(x_0,t_0)$ and all $t_0\le\tau\le t_0+\delta$
\[g(\phi(\tau), \tau) \ge \psi(x_0,t_0) + \frac{\theta_1}{2} = V^+(x_0,t_0) + \frac{\theta_1}{2}.\]
Then the dynamic programming principle \eqref{eq:DPP_upper} yields
\begin{small}
\[\begin{aligned}V^+(x_0,t_0) &\ge \sup_{\beta\in\Lambda_t}\inf_{a\in\mathbb{A}_t}\Big\{&&\hspace{-1cm}\min \Big[  \min_{\tau\in[t_0,t_0+\delta]}\max_{s\in[t_0,\tau]}g(\phi_{x_0,t_0}^{a,\beta[a]}(s),s),\\&&&\hspace{-0.25cm}\max_{\tau\in[t_0,t_0+\delta]}g(\phi_{x_0,t_0}^{a,\beta[a]}(\tau),\tau)\Big]\Big\}\\
&\ge V^+(x_0,t_0) + \frac{\theta_1}{2},
\end{aligned} \]
\end{small}\noindent
which is a contradiction, as $\theta_1>0$.

If, on the other hand, \eqref{eq:not_super_2} holds, then there is a small enough $\delta>0$ such that
\[l(\phi(\tau), \tau) \le\psi(x_0,t_0) + \frac{\theta_2}{2} = V^+(x_0,t_0) + \frac{\theta_2}{2},\]
and, for some strategy $\beta\in\Lambda(t_0)$ and all inputs $a\in\mathbb{A}(t_0)$,
\[\begin{split}\frac{\theta_3}{2}\delta &\le \psi(\phi_{x_0,t_0}^{a,\beta[a]}(t_0+\delta), t_0+\delta) - \psi(x_0,t_0) \\&\le V^+(\phi_{x_0,t_0}^{a,\beta[a]}(t_0+\delta), t_0+\delta) - V^+(x_0,t_0),\end{split}\]
with the latter inequality obtained from \eqref{eq:not_super_2b} by integration on $[t_0,t_0+\delta]$. With this, \eqref{eq:DPP_upper} gives
\begin{small}
\[\begin{aligned}V^+(x_0,t_0) &\ge \sup_{\beta\in\Lambda_t}\inf_{a\in\mathbb{A}_t}\Big\{&&\hspace{-2.5cm}\min \Big[  \min_{\tau\in[t_0,t_0+\delta]}l(\phi_{x_0,t_0}^{a,\beta[a]}(\tau),\tau),\\&&&\hspace{-1.6cm}V^+(\phi_{x_0,t_0}^{a,\beta[a]}(t_0+\delta),t_0+\delta)\Big]\Big\}\\
&\ge V^+(x_0,t_0) + \min\Big\{\frac{\theta_2}{2}, \frac{\theta_3}{2}\delta\Big\},
\end{aligned} \]
\end{small}\noindent
resulting in another contradiction, as $\theta_2,\theta_3,\delta>0$. We therefore conclude that it must be that \eqref{eq:supersolution} holds and $V^+$ is a supersolution of \eqref{eq:HJI_variational}.

Since we have shown that $V^+$ is both a viscosoty subsolution and a viscosity supersolution of the variational inequality, this completes the proof that $V^+$ is a viscosity solution of \eqref{eq:HJI} with Hamiltonian $H^+$. Uniqueness follows from the classical comparison and uniqueness theorems for viscosity solutions (see Theorem 4.2 in \cite{Barron1989}).\qed

\end{proof}
%\end{theorem}
\begin{remark}
As in previous work \cite{Evans1984,Barron1989,Barron1990}, the assumptions in this theorem are stronger than necessary, and can be relaxed. The assumption of Lipschitz continuity of $f$, $l$ and $g$ leads to Lipschitz continuity of $V^\pm$, which then satisfies \eqref{eq:HJI} almost everywhere; if we only assumed uniform continuity, then $V^\pm$ would still satisfy \eqref{eq:HJI} in the viscosity sense. %We refer the interested reader to \cite{Ishii1988,Fleming2006} for various extensions under weaker hypotheses.
\end{remark}

% Solution methodology (1.5-2p)
%% Variational inequality to be solved
%% Proof (Or put this at the end of document)

\section{Numerical Implementation} \label{sec:numerical}
We present in this section a numerical method to compute the value function \eqref{eq:Values} for the time-varying reach-avoid problem, based on the result in Theorem \ref{DOHJI}. For conciseness, we drop the distinction between upper and lower values and Hamiltonians, as the method is equally applicable to either.

Let $\mathbf{i}\in I$ denote the index of the grid point in a discretized computational domain of a compact subset $\X\subset\RR^n$ and let ${k}\in\{1,...,n\}$ denote the index of each discrete time step in a finite interval $[0,T]$. Since our computation will proceed in backward time, we will let $T=t_0>t_1>...>t_n = 0$. To numerically solve the variational inequality (\ref{eq:HJI}), we use the following procedure, based on a three-step update rule:
%
% \begin{align}\label{eq:update}
% \hat{V}(x_\mathbf{i},t_n) &\leftarrow \min \left\{ \{ D_t \hat{V} + \hat{H}(x_\mathbf{i}, D^+_x\hat{V}, D^-_x\hat{V}), l(x_\mathbf{i},t_n) \right\}=0 \}, \notag\\
% \hat{V}(x_\mathbf{i},t_n) &\leftarrow \max\left\{\hat{V}(x_\mathbf{i},t_n), g(x_\mathbf{i},t_n)\right\}.
% \end{align}

%\begin{algorithm}[h] 
%\KwData{$\hat{l}(x_\mathbf{i},t_{k}), \hat{g}(x_\mathbf{i},t_{k})$}
% \KwResult{$\hat{V}(x_\mathbf{i},t_{k})$}
% \BlankLine
% Initialization\DontPrintSemicolon\;\PrintSemicolon
%   \For{$\mathbf{i}\in I$}{
%     \nlset{Init}
%     $\hat{V}(x_\mathbf{i},t_0) \leftarrow \max\{\hat{l}(x_\mathbf{i},t_\mathbf{0}), \hat{g}(x_\mathbf{i},t_\mathbf{0})\}$\;
%   }
%   \BlankLine
%   Value propagation\DontPrintSemicolon\;\PrintSemicolon
%  \For{${k}\leftarrow 1$ \KwTo $n$}{
%     \For{$\mathbf{i}\in I$}{
%     \nlset{U1}
%     $\hat{V}(x_\mathbf{i},t_{k}) \leftarrow \hat{V}(x_\mathbf{i},t_{{k}-1})$ \DontPrintSemicolon\;\PrintSemicolon$\qquad \;\;\;\;+\displaystyle\int_{t_{{k}}}^{t_{{k}-1}} \!\!\!\hat{H}\big(x_\mathbf{i}, D^+_x\hat{V}(x_\mathbf{i},\tau), D^-_x\hat{V}(x_\mathbf{i},\tau)\big)d\tau$\;
%\nlset{U2}
%$\hat{V}(x_\mathbf{i},t_{k}) \leftarrow \min \left\{\hat{V}(x_\mathbf{i},t_{k}), l(x_\mathbf{i},t_{k})\right\}$\;
%\nlset{U3}
%$\hat{V}(x_\mathbf{i},t_{k}) \leftarrow \max\left\{\hat{V}(x_\mathbf{i},t_{k}), g(x_\mathbf{i},t_{k})\right\}$\;
%     }
%   }
% \caption{Numerical Double-Obstacle HJI Solution\label{alg:HJI}}
%\end{algorithm}
%
% --------- surrogate fig for algorithm
\begin{figure}[!h]
	\centering
	\includegraphics[trim=15.5mm 135mm 108mm 70mm, clip=true,width=0.5\textwidth]{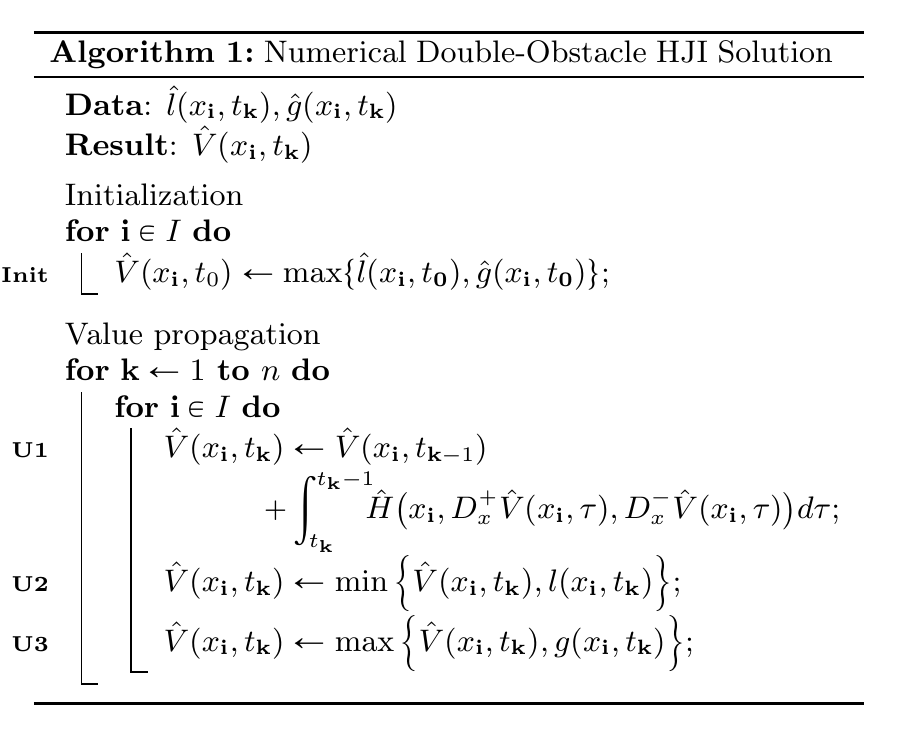}
\end{figure}
The method uses discretized values of the payoff function $\hat{l}(x_\mathbf{i},t_{k})$ and the discriminator function $\hat{g}(x_\mathbf{i},t_{k})$; $\hat{V}$ denotes the numerical approximation to $V$. The integral in the first update step (U1) is computed numerically using time derivative approximations. As an illustrative example, with a first order forward Euler scheme, we would have
\begin{align}
\hat{V}(x_\mathbf{i},t_{k}) &=\; \hat{V} (x_\mathbf{i},t_{{k}-1})\\
&+(t_{{k}-1}-t_{{k}})\hat{H}\!\big(x_\mathbf{i}, D^+_x\hat{V}(x_\mathbf{i},t_{{k}-1}), D^-_x\hat{V}(x_\mathbf{i},t_{{k}-1})\big). \notag
\end{align}
The numerical scheme of Algorithm 
%1
\ref{alg:HJI}
is consistent with (\ref{eq:HJI}). $D^+_x\hat{V}, D^-_x\hat{V}$ represent the ``right" and ``left" approximations of spatial derivatives. For the numerical Hamiltonian $\hat{H}$, we use the Lax-Friedrich approximation \cite{osher91, mitchell-thesis}:
\begin{equation}
\begin{aligned}
\hat{H}(x_\mathbf{i}, D^+_x\hat{V}, D^-_x\hat{V}) =& \;H\left(x_\mathbf{i}, \frac{D^-_x\hat{V}+D^+_x\hat{V}}{2}\right) \\
&-\frac{1}{2} \alpha^\top (D^+_x\hat{V}-D^-_x\hat{V}).
\end{aligned}
\end{equation}
The components of $\alpha$ are given by $\alpha_i = \max_{p\in\mathcal{I}} \big|\frac{\partial H}{\partial p_i} \big|$, where $\mathcal{I}$ is a hypercube containing all the values that $p$ takes over the computational domain. With this choice of $\alpha$ for the Hamiltonian, the numerical scheme is stable \cite{osher91, mitchell-thesis}.

In the numerical examples in Section \ref{sec:example}, we use a fifth-order accurate weighted essentially non-oscillatory scheme \cite{osher91,osher03} for the spatial derivatives $D^\pm_x\hat{V}$; for the time derivative $D_t \hat{V}$, we use a third-order accurate total variation diminishing Runge-Kutta scheme \cite{osher03, shu88}. These methods are implemented by means of the computational tools provided in \cite{Mitchell2004}. It should be noted that lower order spatial and time derivative approximations can also yield a numerically stable (although less accurate) solution to (\ref{eq:HJI}) at lower computational expense \cite{osher91, mitchell-thesis}.

It is important to stress the remarkable computational similarity of this new method to its time-invariant counterpart. Indeed, the only computational overhead is introduced by step (U3) in Algorithm
%1,
 \ref{alg:HJI},
and the need to allow functions $l$, $g$ and $\hat{H}$ to depend on time. As a result, as will be demonstrated in the following section, our method can compute the backwards reachable set for time-varying problems at essentially no additional cost compared to the time-invariant case.

Lastly, the optimal action for each player is implicitly obtained in solving the minimax to compute the Hamiltonian $\hat{H}$ in step (U1). It follows from Algorithm \ref{alg:HJI} that, starting inside a player's winning region (the reach-avoid set for the attacker and its complement for the defender), applying this optimal action at each state as a feedback policy yields a guaranteed winning strategy for the reach-avoid game.

%---------------------------The following discussion is moved to the discussion:

%This sets apart our approach from previously proposed methods that work around time variation by incorporating time as an additional variable in the state: these algorithms require ``two times", one for propagation of value along trajectories (analogously to time-invariant methods) and another for parametrizing time-dependence of dynamics and sets. The key contribution of our method is the ability to leverage the \emph{same} time dimension to simultaneously parameterize the system and propagate the value of the game, thereby achieving a substantial reduction in computational cost (a problem with, say, 1000 time steps becomes roughly 1000 times cheaper to solve). In many important application contexts, such as online safety analysis in dynamical systems \cite{Akametalu2014}, this improvement can allow timely obtention of results that would otherwise entail an impractical computational effort.
% Numerical Implementation (0.5-1p)
%% Discretization schemes and update rule

\section{Numerical Examples }\label{sec:example}
To illustrate our proposed method for computing reach-avoid sets, we present two numerical examples. The first shows the computational procedure in a simple optimal control scenario with a moving target and a moving obstacle, and the obtained capture basin is validated against the analytical result. The second example presents a two-player reach-avoid game with moving target and constraint sets; our method is benchmarked against the approach proposed in \cite{Bokanowski2011}, reaching the same computed set (within one grid cell of accuracy) at substantially lower computational cost.

\subsection{Example 1: Reachability Problem}
Consider the simple optimal control problem below, consisting of a vehicle that can move in any direction at some maximum speed trying to reach a moving target set while avoiding a moving  obstacle, modeled as the complement of the constraint set. The time span of the problem is $[0,T]$ with $T=0.5$. The system is described by the state vector $\mathbf{x}(t)=(\x(t), \y(t))$, which represents the vehicle's position on a plane, with the following dynamics:
\begin{equation}
\label{eq:dyn}
\dot{\mathbf{x}}=v_{\text{veh}}u(t), \quad u(t)\in \mathcal{U},
\end{equation}
where $v_{\text{veh}}=0.5$ is the speed of the vehicle and $\mathcal{U}$ is the unit disk.

The target set is a square with side length $0.4$ moving down in forward time with velocity $v_{\text{tar}}=1.5$. The center of the target set is at $(0, 0.75)$ at $t=0$, and $(0, 0)$ at $t=0.5$. The set is given mathematically as follows:
\begin{equation}
\begin{aligned}
\mathcal{T}_t &= \{(\x,\y): \max (|\x|,|\y-0.75 + v_{\text{tar}} t| )\le0.2 \}.\\
%\mathcal{T}_0 &= \{(\x,\y): \max (|\x|,|\y| )\le0.2 \}
\end{aligned}
\end{equation}
We represent this moving target set using a signed distance function, $l(\x,t)$: $l(\x,t)\le 0 \iff \x\in \mathcal{T}_t$.

The obstacle is a square with side length $0.2$ moving down (in forward time) with velocity $v_{\text{obs}}=1$. The center of the obstacle is at $(0, 0)$ at $t=0$, and $(0, -0.5)$ at $t=0.5$. The set is expressed as:
\begin{equation}
\begin{aligned}
\mathcal{O}_t &= \{(\x,\y): \max (|\x|,|\y + v_{\text{obs}}t| )<0.1 \}. \\
%\mathcal{O}_0 &= \{(\x,\y): \max (|\x|,|\y | )<0.1 \}
\end{aligned}
\end{equation}
We represent this moving obstacle by defining the constraint set $\K_t=\RR^2\setminus\mathcal{O}_t$, through the discriminator function $g(\x,t)$: $g(\x,t)\le 0 \iff \x\notin \mathcal{O}_t$.

%\subsection{Numerical Solution}
Figure \ref{fig:reach_time} shows the time evolution of the numerical solution for the example problem described above. %The target set is shown in green, the   obstacle is shown in black, and the numerically computed reach-avoid set boundary is shown in blue.

The obtained reach-avoid set is not unlike what one might expect out of intuition: because the target set is moving down at a speed greater than that of the vehicle, the lower boundary of the capture basin for $t=0.45$ consists of states from which the vehicle can meet the target set at its final position. This lower boundary directly below the target moves down in backward time (as the vehicle has more time to get to this final position), but eventually gets ``blocked" by the obstacle ($t=0.3$). For earlier times ($t=0.1$), the boundary is ``pinched inwards" again, including nearby states from which the vehicle can move around the obstacle to get to the target; yet, there remains a triangular region directly below the obstacle, shown in the $t=0$ subplot, that is not part of the reach-avoid set, because starting from those states the vehicle is unable to avoid the obstacle that is moving down. The diagonal boundaries of the capture basin at its upper region are formed by those states from which the vehicle can meet the target set between its initial and final positions.  Lastly, the target set itself is part of the reach-avoid set, since a vehicle starting inside the target (which in turn is fully inside the constraint set) has immediately succeeded in reaching it through an admissible trajectory.
\begin{figure}
	\centering
	\includegraphics[trim=19mm 3mm 10mm 10mm, clip=true,width=0.5\textwidth]{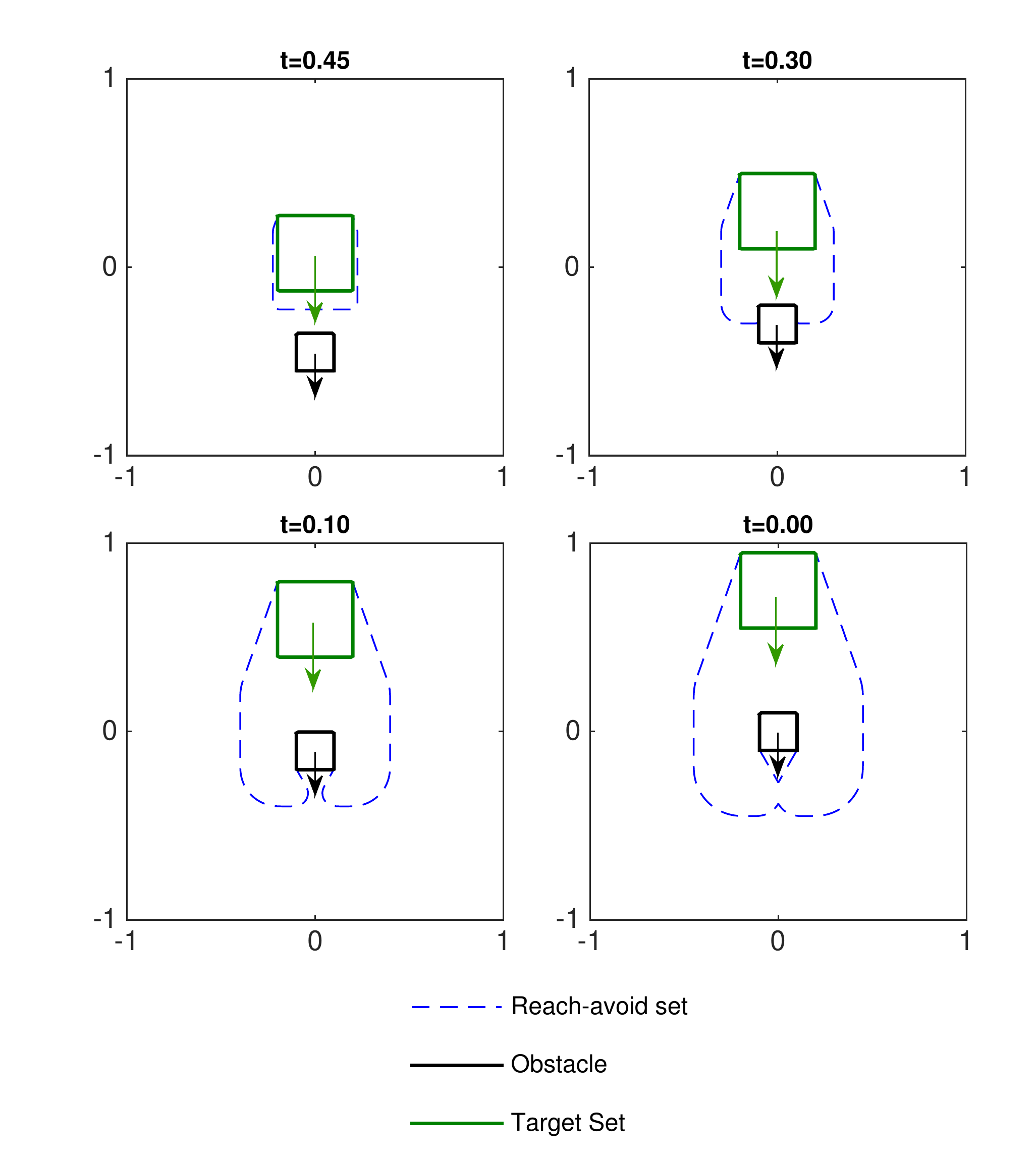}
	\caption{Time evolution of the reach-avoid set for a problem with a target (large square) moving down at speed 1.5, and an obstacle (small square) moving down at speed 1. The inside of the dashed boundary represents the set of states that can reach the target set while avoiding the obstacle.}
	\label{fig:reach_time}
\end{figure}

\subsubsection{Analytic Solution }\label{subsec:analytic}
The reach-avoid set boundary for this example problem can be computed analytically, and thereby compared against the numerically obtained boundary. Because the problem is symmetric about the $\y$ axis, we will consider the capture basin in the region $\x\le0$. We now derive the analytic boundary by considering several different segments separately; it will convenient to refer to Figure \ref{fig:reach} below.

The optimal path for a vehicle with initial position on segment \ref{seg:top} of the capture basin boundary is a straight trajectory, perpendicular to the segment, that reaches the upper corner of the moving target at some intermediate position.  Segment \ref{seg:top} is continued by a short arc \ref{seg:lostarc} comprising initial states from which the vehicle can follow a straight path reaching this top corner exactly at the target's final position. For a vehicle starting on segments \ref{seg:sideleft}, \ref{seg:bottomleft} and \ref{seg:bottomflat} the optimal action is to take the shortest path to the closest point of the target's final position, which will be reached at exactly the final time. The optimal action for a vehicle with initial position on segment \ref{seg:bottomobs} is similar: it must follow a straight line to barely miss the obstacle, before redirecting its path to the target, reaching it at the final time. Finally, a vehicle initially within the triangular region enclosed by segment \ref{seg:obs} and the obstacle cannot avoid being hit by the obstacle. Based on these considerations, the expression for each of these segments can be geometrically derived, leading to the following:

\begin{figure}[h]
	\centering
	\includegraphics[trim=10mm 59mm 0mm 33mm, clip=true,width=0.5\textwidth]{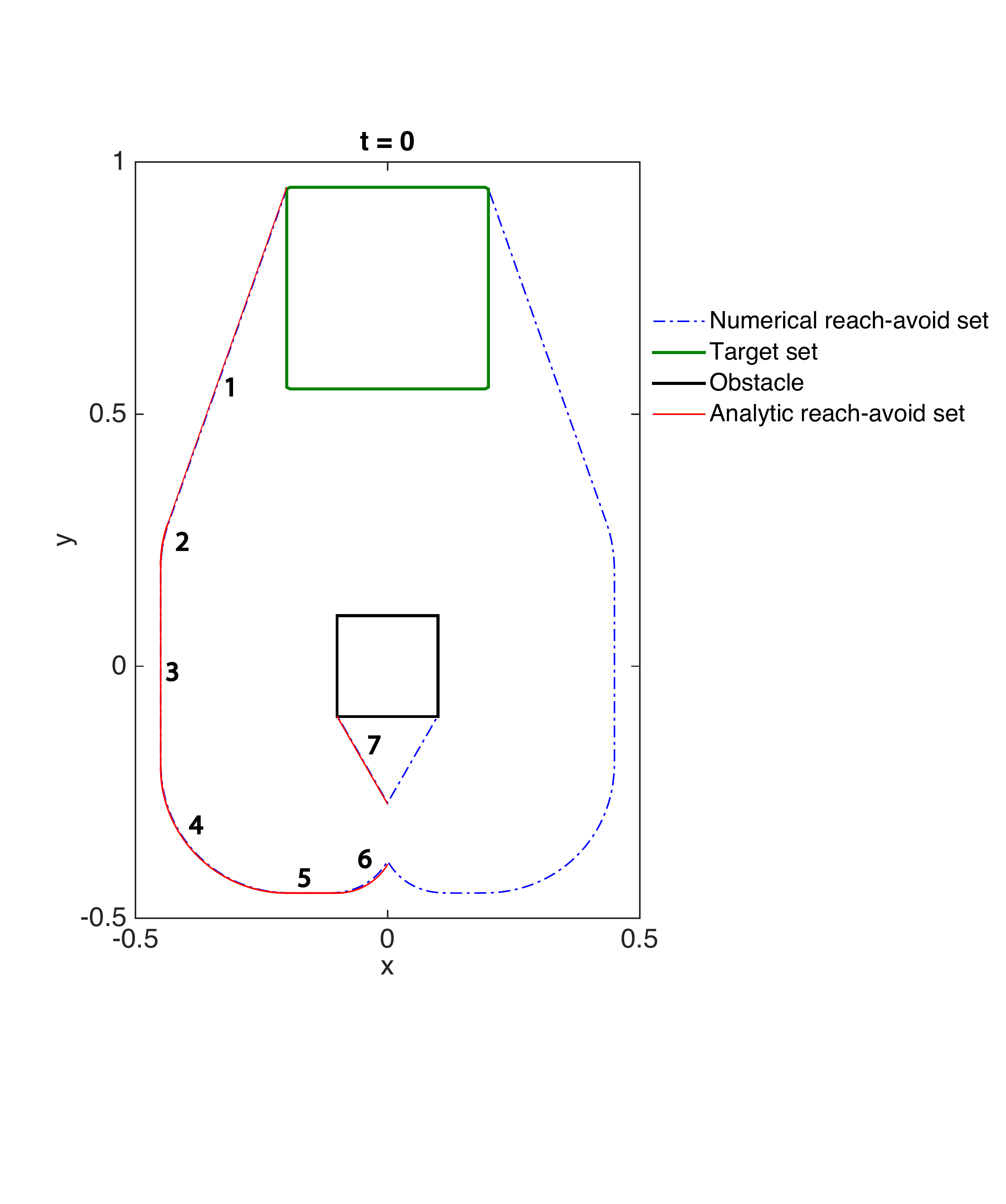}
	\caption{Analytic and numeric reach-avoid set.}
	\label{fig:reach}
\end{figure}

\begin{enumerate}
\item Upper diagonal segment:%, $\y\in[0.2, 0.95]$
\label{seg:top}
\begin{equation}
\{(\x, \y): \y = m(\x + 0.2) + 0.95, \x \in [\x^*, -0.2] \},
\end{equation}
where
\begin{equation}
\begin{aligned}
\x^*&=\frac{-v_{\text{tar}}T}{m^{-1}+m}-0.2,\\
\y^*&=\frac{-v_{\text{tar}}T}{m^{-2}+1}+0.95,\\
m &= \sqrt{\frac{v_{\text{tar}}^2-v_{\text{veh}}^2}{v_{\text{veh}}^2}}.
\end{aligned}
\end{equation}
\item Upper transition arc \label{seg:lostarc} 
\begin{equation}
\begin{aligned}
&\{(\x, \y): (\x+0.2)^2 + (\y-0.2)^2 = \left(v_{\text{veh}}T\right)^2, \\
&\qquad\x\in[-0.45, \x^*], \y\in[0.2, \y^*]   \} ,
\end{aligned}
\end{equation}
%
%The boundary of the reach-avoid set consists of a line segment of slope $m$ starting from $p_{tc}$ and ending at $(\x,\y)$. Then, starting at $(\x,\y)$, we have an arc of a circle centered at $p_{tc} - \left(0, v_{\text{tar}}T\right)$ that ends at $p_{tc} - \left(v_{\text{veh}}T, v_{\text{tar}}T\right)$.
%\todo[inline]{Write this mathematically as a set.}
\item Side straight segment:% $$\y\in[-0.2, 0.2]$:
\label{seg:sideleft}
%
%In this segment of the reach-avoid set boundary, the optimal control of the vehicle is to move in a straight line to the right at maximum speed. The resulting trajectory does not meet the obstacle, and enters the target set at its final position. Since the the left side of the target is at $x=-0.2$, while the vehicle can travel for a distance of $0.25$ in the time horizon, the boundary is given analytically by
%
\begin{equation}
\{(\x,\y): \y\in[-0.2, 0.2], \x=-0.45\}
\end{equation}
\item Outer bottom rounded corner:\label{seg:bottomleft}
%
%In this segment of the reach-avoid set boundary, the optimal control of the vehicle is to again move in a straight line at maximum speed. The resulting trajectory does not meet the obstacle, and enters the lower left corner of the target set at its final position. Since at the end of the time horizon, the lower left corner of the target arrives at $(\x,\y)=(-0.2,-0.2)$, while the vehicle can travel for a distance of $0.25$, the boundary is given analytically by
%
\begin{equation}
\begin{aligned}
\{(\x,\y): \x = \cos(\theta)-0.2, \y = \sin(\theta)-0.2, \\ \theta\in[\pi, 3\pi/2]\}
\end{aligned}
\end{equation}
\item Bottom straight segment:% $\x\in[-0.2, -0.1]$
\label{seg:bottomflat}
%
%In this segment of the reach-avoid set boundary, the optimal control of the vehicle is to move up at maximum speed. The resulting trajectory does not meet the obstacle, and enters the target set at its final position. Since at the end of the time horizon, the bottom of the target arrives at $x=-0.2$, while the vehicle can travel for a distance of $0.25$, the boundary is given analytically by
%
\begin{equation}
\{(\x,\y): \x\in[-0.2, -0.1], \y=-0.45\}
\end{equation}
\item Bottom rounded corner under obstacle:%, $\x\in[-0.1,0]$
\label{seg:bottomobs}
\begin{equation}
\{(\x,\y): \y = y_{om} - \sqrt{d^2 - (\x-x_{om})^2}, \x\in[x_{om},0]\}
\end{equation}
where $(d,y_{om})$ solves% (\ref{eq:pinch}).
\begin{equation} \label{eq:pinch}
\begin{aligned}
d + y_{tf} -y_{om} &= C \\
\frac{d}{v_{\text{veh}}} &= \frac{y_{oi}-y_{om}}{v_{\text{obs}}} \\
\end{aligned}
\end{equation}
\item Obstacle's ``shadow": \label{seg:obs}
\begin{equation}
\{(\x, \y): \y = m\x + 0.95, \x \in [-0.1, 0] \}, \\
\end{equation}
where
\begin{equation}
m = \sqrt{\frac{v_{\text{obs}}^2-v_{\text{veh}}^2}{v_{\text{veh}}^2}}.
\end{equation}
%Since the obstacle is moving downwards, some states will not be able to avoid being hit by the obstacle. The resulting reach-avoid set boundary is given by line segments of slopes $\mp m$ where $m$ is determined in a similar fashion as above, but using the obstacle's speed. The line segments start at the lower corners of the initial position of the obstacle, and meet at $\x=0$.
%\todo[inline]{Write this mathematically as a set.}
\end{enumerate}

%\begin{figure}
%\centering
%	\begin{subfigure}{0.18\textwidth}
%	\centering
%	\includegraphics[width=\textwidth]{"fig/analytic"}
%	\caption{Computing optimal trajectory that barely avoids the   obstacle.}
%	\label{subfig:analytic}
%	\end{subfigure}%
%	~
%	\begin{subfigure}{0.25\textwidth}
%	\centering
%	\includegraphics[width=\textwidth]{"fig/analytic2"}
%	\caption{Computing reach-avoid set boundary in the case where trajectories meet the target at some position between the target's initial and final positions.}
%	\label{subfig:analytic2}
%	\end{subfigure}
%\caption{Computation of analytic reach-avoid set boundary for segments \ref{seg:bottomobs} and \ref{seg:top}.}
%\label{fig:analytic}
%\end{figure}

%The left half of the analytic reach-avoid set boundary is shown in Figure \ref{fig:reach}. %The analytic and numerical boundaries are practically indistinguishable to the naked eye. %In Section \ref{subsec:convergence}, we will 
%We will in fact 
%see that the maximum distance between the two boundaries is less than the size of a numerical grid cell.

\subsubsection{Convergence} \label{subsec:convergence}
Using the scheme described in Section \ref{sec:numerical}, we numerically solved the double-obstacle Hamilton-Jacobi variational inequality (\ref{eq:HJI}) on a computation domain consisting of $N\times N$ grid points for $N=51,101,151,201,251,301$. We compared each of the numerical solutions to the analytic solution derived in Section \ref{subsec:analytic} by the following procedure:

\begin{enumerate}
\item Construct signed distance functions with the zero level set corresponding to the boundary of the numerically computed reach-avoid set (for instance using \cite{Mitchell2004}).
\item Evaluate the signed distance functions at approximately $20\,000$ points distributed on the analytically determined boundary of the reach-avoid set.
\end{enumerate}

The values of the signed distance function correspond to the distance between the analytically computed reach-avoid set boundary points to the numerically computed boundary. These values are used as the error metric for the numerical approximation. 

Figure \ref{fig:convergence} shows in logarithmic scale the mean error and maximum error over all analytic points plotted against the number of grid points per dimension.
An additional line is provided to give the scale of error in terms of the size of spatial discretization or grid spacing. Consistently across the different grid spacings, the mean error is approximately one-tenth of the grid spacing, and the maximum error is approximately half of the grid spacing. The numerical scheme therefore converges both in terms of the mean error and the maximum error. 
\begin{figure}
	\centering
	\includegraphics[trim=7mm 3mm 10mm 7mm, clip=true,width=0.5\textwidth]{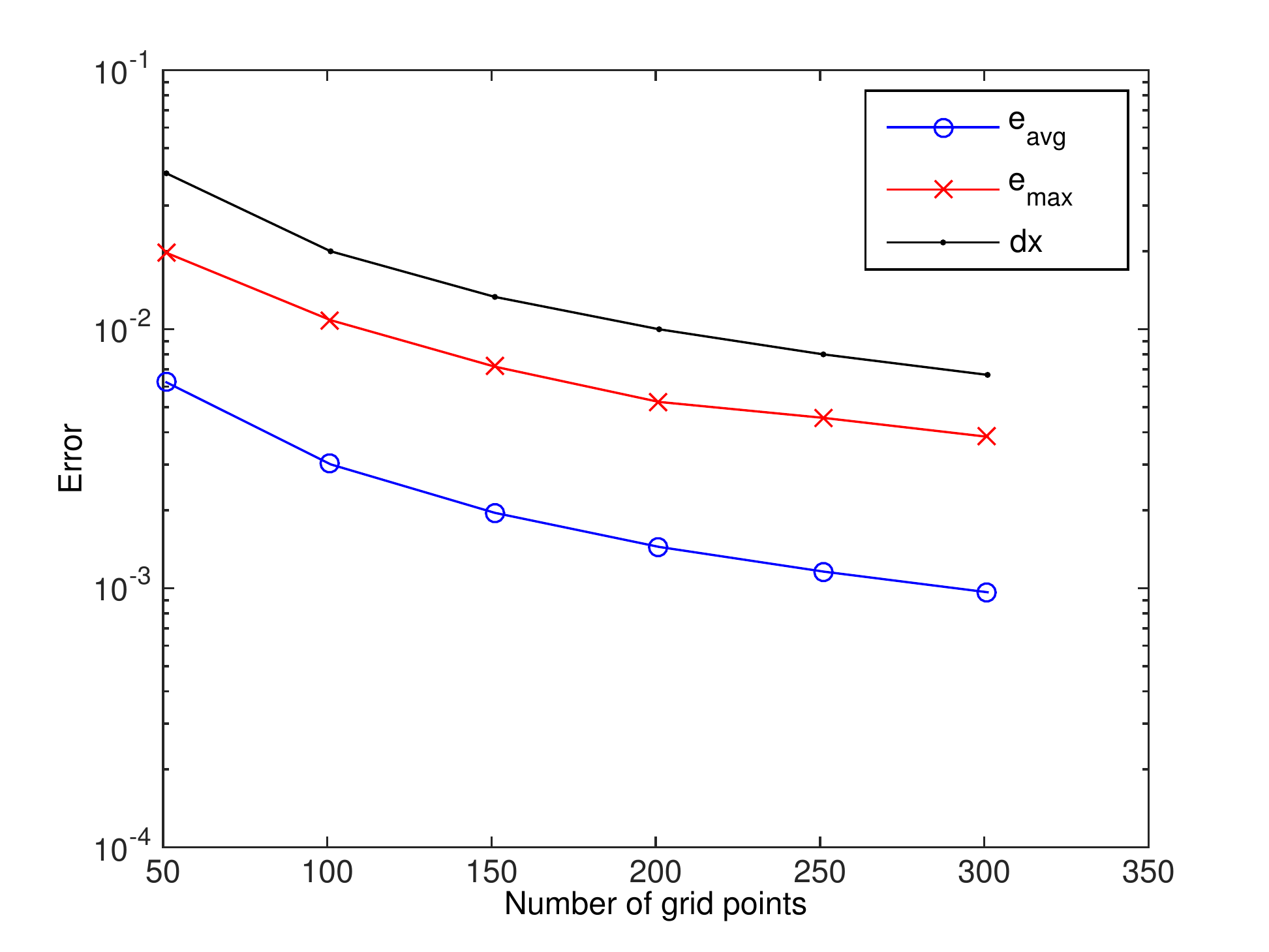}
	\caption{Convergence of our numerical implementation for Example 1 with different grids. Average error is consistently an order of magnitude smaller than the grid spacing, with the maximum error being roughly half of the grid size.}
	\label{fig:convergence}
\end{figure}

\subsection{Example 2: Reach-Avoid Game}
Consider a reach-avoid game in which the attacker  moves in a two-dimensional space while the defender moves on the vertical line $x=0.05$. Let $p_A=(x_A, y_A)$ be the position of the attacker, and $y_D$ be the position of the defender, with $\mathbf{x} = (x_A, y_A, y_D)$ the state of the system, governed by the following dynamics:
\begin{equation}
\begin{split}
\dot{p}_A &= v_A a(t), \quad \|a\|_2\le 1,\\
\dot{y}_D &= v_D b(t), \quad b\in [-1, 1].
\end{split}
\end{equation}
In this reach-avoid game, the attacker wishes to reach a target set that is moving upwards at speed $1.5$, while the defender tries to prevent the attacker from succeeding by intercepting or delaying its advance. The attacker is free to move in any direction at up to some maximum speed, anywhere in a square domain with the exception of a growing obstacle whose lower edge is expanding downwards at a speed of $0.5$. The players have maximum speeds of $v_A = 2, v_D = 3$. Here, interception is defined as the two players being within a radius of $0.1$ of each other. Figure \ref{fig:reachAvoid} shows the initial configuration of the moving target and the moving obstacle, as well the interception set centered at four different defender positions.

For this reach-avoid game, we seek to compute the reach-avoid set, comprised by the set of joint positions from which the attacker is guaranteed to be able to reach the target while avoiding interception by the defender as well as collision with the obstacle. To compute the reach-avoid set, we solve (\ref{eq:HJI}) with the following Hamiltonian:
\begin{equation}
H(\mathbf{x}, \nabla_{\mathbf{x}} V,t) = \min_{a\in \mathcal{A}} \max_{b\in[-1, 1]}\nabla_{p_A} V v_A a(t) + \nabla_{y_D} V v_D b(t).
\end{equation}
Solving the minimax in the Hamiltonian, we get
\begin{equation}
H(x, \nabla_{\mathbf{x}} V,t) = -v_A\|\nabla_{p_A} V\|_2 + v_D |\nabla_{y_D} V|.
\end{equation}

Since the state space of the reach-avoid game is three-dimensional, we visualize two-dimensional cross sections of the three-dimensional reach-avoid set at $t=1$, taken at various defender initial positions. Figure \ref{fig:reachAvoid} compares the two-dimensional slices of the reach-avoid sets computed by the existing state augmentation method and by our newly proposed augmentation-free method. Given the defender positions shown in each of the subplots, the attacker will be able to reach the target if it is on the side of the reach-avoid set boundary containing the target. As can be appreciated, the capture basin boundaries computed by the two methods are very similar (well within a grid cell of distance); however, computation\footnote{Computations were run using \cite{Mitchell2004} on a Lenovo T420s laptop with a Core i7-2640M processor.} using the state augmentation method took approximately 1 hour and 50 minutes on a $45^4$ grid. With our proposed augmentation-free method, computation only took approximately 3 minutes on a $51^3$ grid. Our computation was two orders of magnitude faster and provided essentially the same results.

One can clearly see the effect of the different defender initial positions on the reach-avoid set. If the defender starts the game near the bottom of the domain (top left subplot), the defender would be able to block the attacker from going through the gap between the bottom edge of the obstacle and that of the domain. Thus we see that the reach-avoid set boundary does not extend into the left half quadrant of the domain. However, in this case, the attacker is free to cross the gap above the top edge of the obstacle, which leads to a large area of the top left quadrant being inside the capture basin.

Similarly, if the defender starts near the top of the domain (bottom right subplot), we see a similar but opposite effect. The reach-avoid set extends into the bottom left quadrant of the domain, but less so than it did in the top left quadrant, due to the fact that the passage under the obstacle is closing and the target is moving away from it: an attacker not starting close enough to the opening will either get blocked out or not be able to make it through in time to reach the target. The remaining two subplots (top right and bottom left) show intermediate defender initial positions; in general we see that the reach-avoid set extends more into the top left quadrant of the domain than into the bottom left quadrant.

Figure \ref{fig:reachAvoidTime} shows the backward time evolution of the reach-avoid set for a single defender position. The subplots show the capture basin at various times. At $t=0.92$, there is a relatively small region in the state space from which the attacker can reach the target by the end of the game ($t=1$), as there is little time left. As the starting time $t$ considered decreases, the attacker has more time to reach the target and thus the reach-avoid set grows; however, this growth is inhibited by both the defender's interception set and by the presence of the obstacle.
% It is important to stress that what is being shown are cross-sections of the three-dimensional capture basin for a fixed position of the defender; however, during the game, the defender \emph{will} move to intercept the attacker, which is what prevents the reach-avoid set from growing into the upper left domain. In the lower domain, given that the target moves considerably faster than the obstacle, the determining factor in the extension of the capture basin is not the closing of the lower passage, but the motion of the target away from it. Indeed, at $t=0.25$, an attacker situated to the left of the obstacle will be unable to reach the target, even though the gap is open.
\begin{figure}
	\centering
	\includegraphics[trim=19mm 5mm 21mm 35mm, clip=true,width=0.425\textwidth]{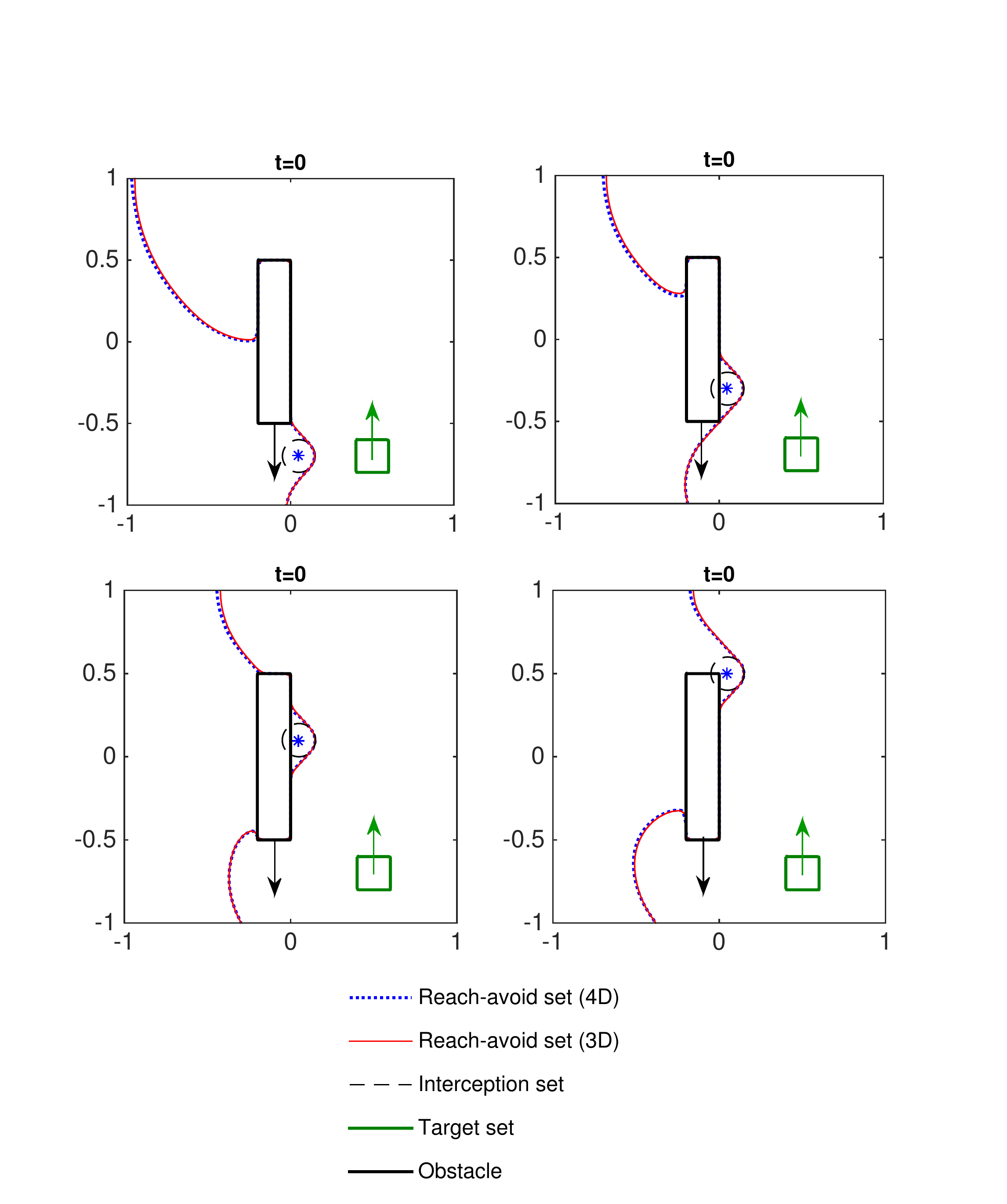}
	\caption{Reach-avoid set computed through the state augmentation method (4D) and our proposed augmentation-free method (3D). 2D cross-sections of the set are shown at the initial time for four different defender positions.}
	\label{fig:reachAvoid}
\end{figure}

\begin{figure}
	\centering
	\includegraphics[trim=10mm 15mm 38mm 7mm, clip=true,width=0.425\textwidth]{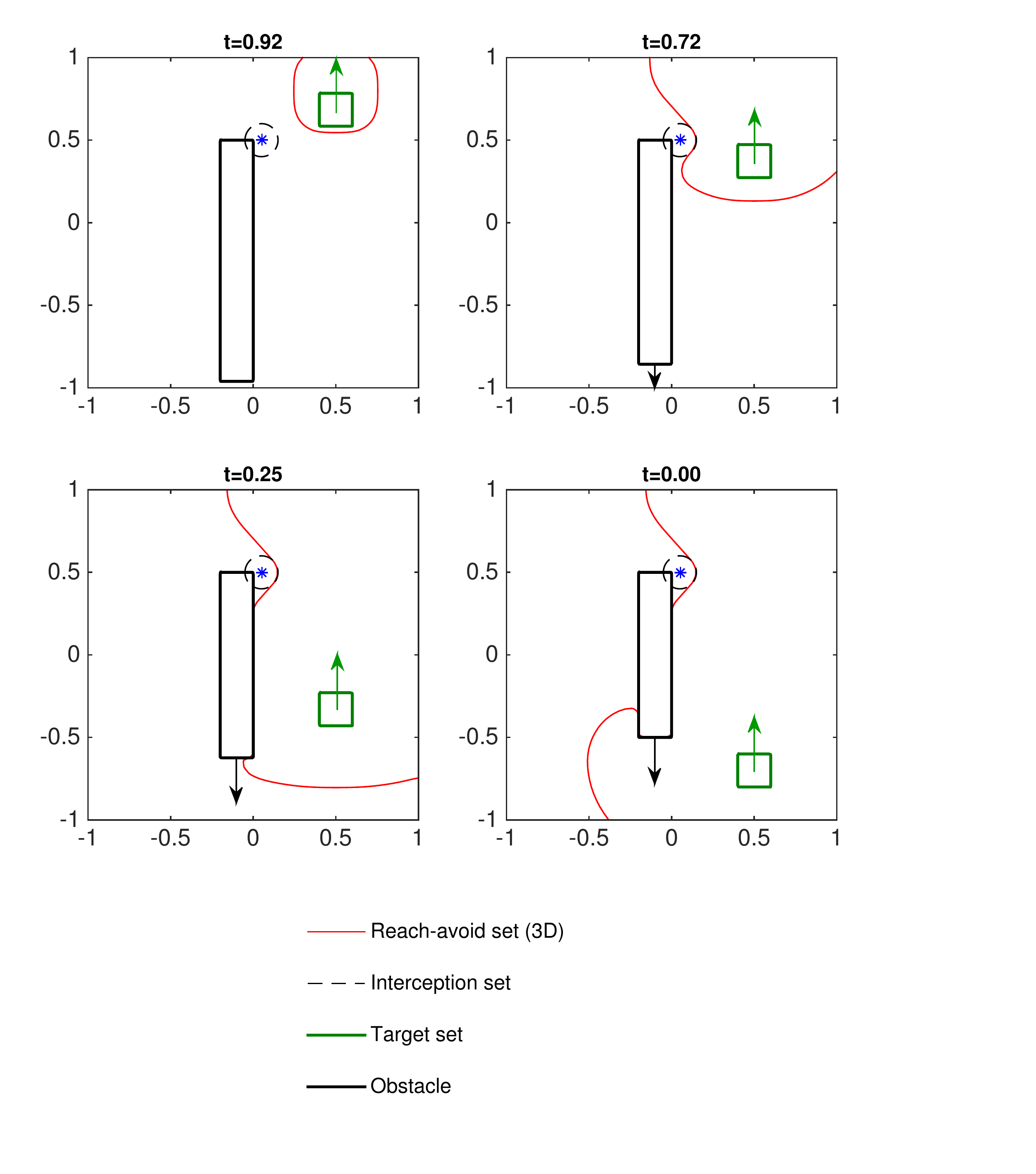}
	\caption{Backward time evolution of the reach-avoid set. As $t$ decreases, the attacker has more time to reach the target, so the reach-avoid set grows. The growth of the reach-avoid set is inhibited by the defender's interception set and the obstacle.}
	\label{fig:reachAvoidTime}
\end{figure}

% Numerical Simulations (1-1.5p)
%% Simple example
%% Analytic solution

\section{Conclusion}
We have presented here a novel extension of Hamilton-Jacobi methods to reach-avoid problems with time-varying dynamics, targets, and constraints. This result allows the analysis of many practically relevant problems in game theory and optimal control, including pursuit-evasion, differential games, and safety certificates for dynamical systems. In particular, our result can provide strong guarantees for collision avoidance in dynamic environments with multiple moving obstacles.

Importantly, numerical implementations of our method have computational complexity equivalent to that of already existing techniques for time-invariant systems.
This sets our method apart from previously proposed approaches that work around time variation by incorporating time as an additional variable in the state. %The key contribution of our technique is the ability to leverage a single time dimension to simultaneously parameterize the system and propagate the value of the game, thereby achieving a substantial reduction in computational cost through the elimination of an unnecessary dimension. % (a problem with, say, 1000 time steps becomes roughly 1000 times cheaper to solve).
In many important application contexts, such as online safety analysis in dynamical systems \cite{Akametalu2014}, the substantial reduction in computational cost introduced by our technique can allow timely obtention of results that would otherwise entail an impractical computational effort.

In the future, we intend to develop applications of this new formulation to large-scale multi-agent systems in both cooperative and adversarial contexts. By leveraging the possibilities of time-varying targets and constraints to encode the trajectories of other agents, we hope to incrementally build efficient solutions that scale linearly, and not exponentially, with the complexity of the multi-agent network.
% Conclusion (0.5p)

%%%%%%%%%%%%%%%%%%%%%%%%%%%%%%%%%%%%%%%%%%%%%%%%%%%%%%%%%%%%%%%%%%%%%%%%%%%%%%%%
%\addtolength{\textheight}{1cm}   % This command serves to balance the column lengths

%\bibliographystyle{plain}
%\bibliography{references,library}
\printbibliography
\end{document}